\documentclass[11pt]{amsart}
\usepackage{graphicx, amssymb, eucal, hyperref, color}
\newtheorem{thm}{Theorem}[section]
\newtheorem{lemma}[thm]{Lemma}
\newtheorem{prop}[thm]{Proposition}
\newtheorem{cor}[thm]{Corollary}
\newtheorem{question}[thm]{Question}
\newtheorem{fact}[thm]{Fact}

\newtheorem*{lemma*}{Lemma}

\theoremstyle{definition}

\newtheorem{nrmk}[thm]{Remark}

\newtheorem*{rmk}{Remark}

\theoremstyle{remark}

\makeatletter

\def\dotminussym#1#2{%
  \setbox0=\hbox{$\m@th#1-$}%
  \kern.5\wd0%
  \hbox to 0pt{\hss\hbox{$\m@th#1-$}\hss}%
  \raise.6\ht0\hbox to 0pt{\hss$\m@th#1.$\hss}%
  \kern.5\wd0}
\newcommand{\dotminus}{\mathbin{\mathpalette\dotminussym{}}}
\renewcommand{\r}{\mathbb{R}}
\newcommand{\Z}{\mathbb{Z}}

\newcommand{\curly}[1]{\mathcal{#1}}

\newcommand{\n}{\mathbb{N}}

\newcommand{\la}{\curly{L}}

\renewcommand{\to}{\rightarrow}

\def \bd{\operatorname{bd}}

\def \<{\langle}
\def \>{\rangle}

\def \*Z {{{^*}\Z}}
\def \((  {(\!(}
\def \)) {)\!)}

\numberwithin{equation}{section}

\def \Th{\operatorname{Th}}

\def \u{\mathcal U}

\def \conn{\operatorname{conn}}

\allowdisplaybreaks[2]

\title{The pseudoarc is a co-existentially closed continuum}
\author{Christopher J. Eagle \and Isaac Goldbring \and Alessandro Vignati}
\thanks{Goldbring's work was partially supported by NSF CAREER grant DMS-1349399.}
\thanks{Eagle's work was partially supported by an Ontario Graduate Scholarship award.}

\address{Department of Mathematics, University of Toronto, 40 St. George Street, Toronto, Ontario, Canada, M5S 2E4}
\curraddr{Department of Mathematics and Statistics, University of Victoria, PO BOX 1700 STN CSC, Victoria, British Columbia, Canada, V8W 2Y2}
\email{eaglec@uvic.ca}
\urladdr{http://www.math.uvic.ca/~eaglec/}
\address {Department of Mathematics, Statistics, and Computer Science, University of Illinois at Chicago, Science and Engineering Offices M/C 249, 851 S. Morgan St., Chicago, IL, 60607-7045}
\email{isaac@math.uic.edu}
\urladdr{http://www.math.uic.edu/~isaac}

\address{Department of Mathematics and Statistics, York University, 4700 Keele Street, Toronto, Ontario, Canada, M3J 1P3}
\email{ale.vignati@gmail.com}
\urladdr{http://www.automorph.net/avignati}


\begin{document}

\begin{abstract}
Answering a question of P. Bankston, we show that the pseudoarc is a co-existentially closed continuum.  We also show that $C(X)$, for $X$ a nondegenerate continuum, can never have quantifier elimination, answering a question of the the first and third named authors and Farah and Kirchberg.
\end{abstract}

\maketitle

\section{Introduction}

A \emph{compactum} is simply a compact Hausdorff space and a \emph{continuum} is a connected compactum.  There has been an extensive study of compacta and continua from the model-theoretic perspective; see, for example, \cite{Steeg2003} or \cite{bankston4}. In \cite{bankston1}, Bankston dualizes the model-theoretic notions of existential embeddings and existentially closed structures to the categories of compacta and continua; the dual notions are (appropriately named) \emph{co-existential mappings} and \emph{co-existentially closed compacta and continua}.  In the appendix to this note, we show how these notions translate to their usual model-theoretic counterparts in the continuous signature for C$^*$-algebras (e.g., $X$ is a co-existentially closed compactum if and only if $C(X)$ is an existentially closed abelian C$^*$-algebra).

Recall that a continuum $X$ is said to be \emph{indecomposable} if $X$ is not the union of two of its proper subcontinua.  If, in addition, every subcontinuum of $X$ is also indecomposable, then $X$ is said to be \emph{hereditarily indecomposable}.  In \cite{bankston3}, Bankston proves that every co-existentially closed continuum is hereditarily indecomposable.

Amongst the hereditarily indecomposable continua, there is one such continuum that plays a special role.  Recall that a continuum $X$ is said to be \emph{chainable} if, for every finite open cover $U_1,\ldots,U_n$ of $X$, there is a refinement to a cover $V_1,\ldots,V_m$ such that $V_i\cap V_j=\emptyset$ if and only if $|i-j|>1$.  Up to homeomorphism, there is a unique continuum that is both chainable and hereditarily indecomposable; this continuum is called the \emph{pseudoarc} (see \cite{Logan}).  In \cite{bankston2}, Bankston asks the natural question:  is the pseudoarc a co-existentially closed continuum?  In this note, we answer Bankston's question in the affirmative.

The plan of the proof is as follows:  motivated by an $\la_{\omega_1,\omega}$ characterization of chainable continua in the (discrete) signature of lattice bases for continua given by Bankston in \cite{bankston5}, we prove that the class of separable $C(X)$ for which $X$ is a chainable continuum is definable by a uniform sequence of universal types (in the terminology of \cite{FHL}) in the continuous signature for C$^*$-algebras.  Together with the fact, proven by K.P. Hart in \cite{hart}, that there is a unique universal theory of $C(X)$ for $X$ a nondegenerate continuum, this allows us to apply the technique of model-theoretic forcing to obtain a metrizable continuum $X$ for which $C(X)$ is existentially closed (so $X$ is hereditarily indecomposable) and for which $X$ is chainable, whence $X$ must be the pseudoarc.

We end this note by observing that $C(X)$, for $X$ a nondegenerate continuum, can never have quantifier elimination.  The proof relies on combining the aforementioned result of Hart together with the observation that the class of $C(X)$, for $X$ a continuum, does not have the amalgamation property, as pointed out to us by Logan Hoehn.  Together with results from \cite{EFKV}, the question of which abelian C$^*$-algebras have quantifier elimination is now completely settled:  the only abelian C$^*$-algebras with quantifier elimination are $\mathbb{C}$, $\mathbb{C}^2$, and $C(2^\n)$.  In \cite{EFKV} it was shown that the only non-abelian C$^*$-algebra with quantifier elimination in $M_2(\mathbb{C})$, so we now have the complete list of C$^*$-algebras with quantifier elimination.

In this paper, we assume the reader is familiar with basic model theory as it applies to C$^*$-algebras.  A good reference for the unacquainted reader is \cite{FHS2}.  For background on the notions appearing in Section 3 (e.g. model companions, model-completions, etc.) we refer the reader to \cite{usvy}.

\subsection{Facts from the model theory of continua}

In this paper, all C$^*$-algebras are unital and we always work in the (continuous) signature $\la$ for (unital) C$^*$-algebras. In this language, the class of abelian C$^*$-algebras is clearly universally axiomatizable.

Throughout this paper, we will apply a topological adjective to an abelian C$^*$-algebra if its Gelfand spectrum possesses that property.  Thus, for example, we will call $C(X)$ connected if $X$ is connected (and thus a continuum).

\begin{fact}
The class of connected abelian C$^*$-algebras is universally axiomatizable.
\end{fact}

\begin{proof}
That the class of connected abelian C$^*$-algebras is closed under ultraproducts is a special case of a more general result due to Gurevic (see \cite[Lemma 10]{gur}).  (Alternatively, if $X$ is compact, then $X$ is connected if and only if $C(X)$ is projectionless; it remains to observe that if $(A_i \mid i\in I)$ is a family of C$^*$-algebras,  $\u$ is an ultrafilter on $I$, and $A=\prod_\u A_i$, then any projection of $A$ can be written as $\pi(p_i)$ with each $p_i$ a projection of $A_i$ and $\pi\colon\prod A_i\to\prod_\u A_i$ the usual quotient map.)  It remains to see that the class is closed under substructures:  if $C(X)\subseteq C(Y)$ and $Y$ is connected, then $Y$ continuously surjects onto $X$, whence $X$ is also connected.
\end{proof}
\begin{nrmk}
We may also give an explicit universal axiomatization of models of abelian projectionless C$^*$-algebras: note that a C$^*$-algebra is abelian if and only if $\sup_{x,y}\|xy-yx\|=0$ and that an abelian C$^*$-algebra is projectionless if and only if \[\sup_{\|f\|=1}\min (2\|1-ff^*\|\dotminus 1,1\dotminus4\|ff^*-(ff^*)^2\|)=0,\] where $r\dotminus s:=\max(r-s,0)$.   To see this, first note that if $p$ is a proper projection, then $\|1-p\|=1$ and $\|p-p^2\|=0$.  Conversely, if $X$ is connected and $f\in C(X)$ with $\|f\|=1$ satisfies $2\|1-ff^*\|\dotminus 1>0$, then the minimum of the spectrum of $ff^*$ is less than $1/2$. In particular, $ff^*$ attains the value $1/2$, whence $\|ff^*-(ff^*)^2\|\geq 1/4$ as required.
\end{nrmk}

We let $T_{\conn}$ denote the $\la$-theory of connected abelian C$^*$-algebras.  The following important fact about $T_{\conn}$ is due to K. P. Hart.  We indicate in the appendix how the version we state here follows from the statement given in \cite{hart}.

\begin{fact}\label{kp}
Suppose that $X$ and $Y$ are continua and $X$ is nondegenerate (that is, $X$ is not a single point).  Then there is $C(X')\equiv C(X)$ for which $C(Y)$ embeds into $C(X')$.
\end{fact}

In particular, any two $C(X)$'s for $X$ a nondegenerate continuum have the same universal theory.  As an aside, a standard ``sandwiching'' argument shows that if $X$ and $Y$ are nondegenerate continua for which both $\Th(C(X))$ and $\Th(C(Y))$ are $\forall\exists$-axiomatizable, then $\Th(C(X))=\Th(C(Y))$.

As another aside, one can use Fact \ref{kp} to prove that every complete theory of nondegenerate continua has continuum many nonisomorphic separable models.  Indeed, first recall \cite[Section 20]{MT} that there is a family $(X_\alpha)_{\alpha<2^{\aleph_0}}$ of nondegenerate metrizable continua such that, for any metrizable continuum $Y$, $Y$ maps onto at most countably many of the $X_\alpha$'s.  Now given a nondegenerate metrizable continuum $Y$ and $\alpha<2^{\aleph_0}$, by Fact \ref{kp} one can find a metrizable continua $Y_\alpha$ such that $C(X_\alpha)\hookrightarrow C(Y_\alpha)$ and $C(Y_\alpha)\equiv C(Y)$.  It remains to observe that continuum many of the $Y_\alpha$'s must be pairwise non-homeomorphic.

\section{Proof of the Main Theorem}

Following the terminology of \cite{FHL} (see also \cite{FarahMagidor}), we say that a class $\mathcal{K}$ of separable models of $T_{\conn}$ is \emph{definable by a uniform sequence of universal types} if there are \emph{existential} $\la$-formulae $(\varphi_{m,n}(\vec x_m) \mid m,n\in \n)$ (taking only nonnegative values) such that a model $C(X)$ of $T_{\conn}$ belongs to $\mathcal{K}$ if and only if, for all $m\in \n$, we have
\[
\left(\sup_{\vec x_m}\inf_n \varphi_{m,n}(\vec x_m)\right)^{C(X)}=0.
\]
\begin{thm}\label{omitting}
The class of models $C(X)$ of $T_{\conn}$ with $X$ chainable is uniformly definable by a sequence of universal types.
\end{thm}

\begin{proof}

Consider the following quantifier-free $\la$-formulae, where $\vec x=(x_1\ldots,x_k)$, $\vec y=(y_1,\ldots,y_m)$, and $\vec z=(z_{ij})_{1\leq i,j\leq m}$:
\begin{itemize}
\item $\psi_0^k(\vec x)=\|\sum |x_{i}|\|-\|(\|\sum |x_i|\|-\sum |x_i|)\|$
\item\label{c3} $\psi_1^{m}(\vec y):=\max_{|i-j|\geq 2}\sqrt{\|y_iy_j\|}$
\item\label{c4} $\psi_2^{k,m}(\vec x,\vec y,\vec z):=\max_j \min_i d(|x_i|-|y_j|,|z_{ij}|)$
\end{itemize}

Note that if $X$ is a continuum and $\vec f\in C(X)^k$, then $\psi_0^k(\vec f)^{C(X)}\geq 0$ and $\psi_0^k(\vec f)^{C(X)}=r$ if and only if $r$ is maximal such that $\operatorname{range}(\sum |f_i|)\subseteq[r,\infty)$. 


Let $\sigma_{k}(\vec x)$ denote the infinitary formula
\begin{eqnarray*}
\inf_{m}\inf_{\vec g}\inf_{\vec h} \min(\psi_0^{k}(\vec x),\max(\psi_0^k(\vec x)\dotminus k\psi_0^m(\vec y),m\psi_1^{m}(\vec y),m\psi_2^{k,m}(\vec x, \vec y,\vec z)))).
\end{eqnarray*}
\noindent To prove the theorem, we show that a metrizable continuum $X$ is chainable if and only if, for all $k$, we have 
\[
\left(\sup_{\vec x}\sigma_{k}(\vec x)\right)^{C(X)}=0.
\]
Since we will evaluate formulas only in $C(X)$ during this proof, we suppress the superscript $C(X)$ to simplify notation.

We start with the  backward implication.  Suppose that $X$ is a metrizable continuum for which $\sup_{\vec x}\sigma_{k}(\vec x)=0$.  Fix an open cover $U_1,\ldots,U_k$ of $X$. For $i=1,\ldots,k$, fix a nonnegative function $f_i$ in $C(X)$ with $\|f_i\|=1$ for which $U_i=\{x\in X\colon f_i(x)\neq 0\}$. (This is possible for, in a metrizable compact space, every closed set is the zeroset of a continuous function.) Since $(U_i)$ forms a cover of $X$, we have $\psi_0^k(\vec f)>0$.  Fix $\epsilon>0$ for which $\psi_0^k(\vec f)>k\epsilon$. Take $m,\vec g$ and $\vec h$ witnessing that $\sigma_k(\vec f)<\frac{\epsilon}{2}$. Without loss of generality, we may assume that each $g_j$ and $h_{ij}$ are nonnegative functions.

Since $\psi_0^k(\vec f)\dot-k\psi_0^m(\vec g)<\frac{\epsilon}{2}$, we have that $\psi_0^m(\vec g)>\frac{\epsilon}{2}$. For $j=1,\ldots,m$, set 
\[
W_j:=\left\{x\mid g_j(x)> \frac{\epsilon}{2m}\right\}.
\]
\noindent Note that $W_1,\ldots,W_m$ covers $X$.  We next show that $(W_j)$ refines $(U_i)$.  Fix $j$ and take $i$ such that $m\cdot d(f_i-g_j,h_{ij})<\frac{\epsilon}{2}$; we show that for such$i$, $W_j\subseteq U_i$.  Towards this end, fix $x\in W_j$.  Since $h_{ij}$ is nonnegative we have
\[-\frac{\epsilon}{2m} < f_i(x) - g_j(x),\]
and from the definition of $W_j$ we have
\[f_i(x)-g_j(x)<f_i(x)-\frac{\epsilon}{2m};\]
that is, $-\frac{\epsilon}{2m} < f_i(x) - \frac{\epsilon}{2m}$, whence $f_i(x)>0$ and $x\in U_i$.

Next fix $i,j\in \{1,\ldots,m\}$ with $|i-j|\geq 2$.  Suppose, towards a contradiction, that $x\in W_i\cap W_j$. Then $\sqrt{g_ig_j(x)}> \frac{\epsilon}{2m}$, whence $m\psi_1^m(\vec g)^{C(X)}> \frac{\epsilon}{2}$, which is a contradiction.  The only thing preventing from $(W_j)$ from being the desired chain refinement of $(U_i)$ is that we may not have that $W_i\cap W_{i+1}\not=\emptyset$ for all $1\leq i<m$.  However, if $W_i\cap W_{i+1}=\emptyset$ for some $i<m$, then by connectedness of $X$ we have that either $X=\bigcup_{k\leq i}W_k$ or $X=\bigcup_{k\geq i+1}W_k$.  We then pass to the appropriate subsequence of $W_j$, noting that the previous verified conditions of $(W_j)$ persist.  This process must end after a finite number of steps, yielding the desired refinement of $(U_i)$.

We now prove the direct implication.  Suppose that $X$ is chainable and fix $\vec f\in C(X)^k$.  We must show that $\sigma_k(\vec f)=0$.  Without loss of generality, we may assume that each $f_i$ is a nonnegative function of norm $1$. If $0\in \operatorname{range}(\sum_i f_i)$ then $\sigma_k(\vec f)=\psi_0^k(\vec f)=0$.  Thus we may suppose that there is $\delta$ with $\psi^k_0(\vec f)>\delta>0$ and choose $\epsilon',\epsilon>0$ with 
\[
\psi_0^k(\vec f)>k\epsilon'>k\epsilon>\psi_0^k(\vec f)-\delta>0.
\]
 Note that $U_i:=\{x\mid f_i(x)>\epsilon\}$ covers $X$. Let $(W_j)_{j\leq m}$ be a chain that is a refinement for $U_i$ with $m$ minimal.  For $j=1,\ldots,m$, it is routine to construct nonnegative functions $g_j\in C(X)$ with the following properties:
\begin{itemize}
\item[(i)] $g_i\leq f_{i(j)}$ where $i(j)$ is the minimum $i$ such that $W_j\subseteq U_i$;
\item[(ii)] for all $x\in W_j$, $g_j(x)>\epsilon$; 
\item[(iii)] $\|g_j\|<\epsilon'$.
\end{itemize}

Note that (i) implies that $\inf_{\vec z}\psi_2^{k,m}(\vec f,\vec g,\vec z)=0$ and (ii) and (iii) imply that $\psi_0^{m}(\vec g)\in [k\epsilon,k\epsilon']$.  We need to some finagling to arrange that $\psi_1^m(\vec g)=0$.

For $Y\subseteq X$, set $\bd(Y):=\overline Y\setminus Y$.
For $1\leq j\leq m$, note that $\bd(W_j)\subseteq W_{j-1}\cup W_{j+1}$ (with the convention $W_0=W_{m+1}=\emptyset$). Moreover, by the connectedness of $X$, if $j\neq 1,m$, we have \[\bd(W_j)\cap W_{j-1}\neq\emptyset\neq \bd(W_j)\cap W_{j+1}.\]

For $j>1$ and $x\in \bd(W_j)\cap W_{j-1}$, we can find (using, for example, the metrizability of $X$) an open neighborhood $U_x$ of $x$ with the property that, setting 
\[
Z_j:=\bigcup_{x\in \bd(W_j)\cap W_{j-1}}U_x,
\]

\noindent we have
\begin{itemize}
\item $Z_j\subseteq W_{j-1}\cup W_j$;
\item $\overline {Z_j}\cap \overline{W_{j+1}}=\emptyset$;
\item $\overline Z_j\cap \overline Z_{j+1}=\emptyset$.
\end{itemize}

\noindent For convenience, we set $Z_1:=\emptyset$ and $Z_{m+1}:=\emptyset$.

\

\noindent For $j=1,\ldots,m$, we consider the open set
\[
N_j:=Z_j\cup Z_{j+1}\cup (W_j\setminus \overline{\bigcup_{i>j}W_i})
\]

\noindent and its (proper) closed subset

\[
M_j:=\overline W_j\cap (X\setminus \bigcup_{i>j} W_i).
\]

We leave it to the reader to verify that $X=\bigcup_{j=1}^m M_j$ and that $N_i\cap N_j=\emptyset$ for $|i-j|\geq 2$.

For $j=1,\ldots,m$, take nonnegative $g_j'\in C(X)$ with $\|g_j'\|=1$ for which $N_j=\{x\in X \ : \ g_j'(x)\not=0\}$ and for which $g_j'$ is identically $1$ on $M_j$.  Finally, set $g_j'':=g_j\cdot g_j'$.  It remains to establish:

\[
\inf_{\vec h}\,\max(\psi_0^k(\vec f)\dotminus k\psi_0^m(\vec g''),\,m\psi_1^{m}(\vec g''),\,m\psi_2^{k,m}(\vec f, \vec g'',\vec h)) \leq\delta. \quad (\dagger)
\] 

\

First observe that since $m$ was chosen to be minimal, there is $z\in W_1\setminus(\overline Z_2\cup W_2)$.  It follows that $\sum g_j''(z)=g_1(z)<\epsilon'$, whence $\psi_0^m(\vec g'')<\epsilon'$.  To obtain a lower bound on $\psi_0^m(\vec g'')$, fix $x\in X$ and take $j$ such that $x\in M_j$.  Then $x\in \overline{W_j}$ whence $g_j(x)\geq \epsilon$ and $g_j'(x)=1$, so $g_j''(x)\geq \epsilon$.  Since $x\in X$ is arbitrary, we have that $\psi_0^m(\vec g)\geq \epsilon$.  It follows that $\psi_0^k(\vec f)\dotminus k\psi_0^m(\vec g'')<\delta$.  

Next observe that since $N_i\cap N_j=\emptyset$ for $|i-j|\geq 2$, we have that $\psi_1^m(\vec g'')=0$.  Finally, since $g_j''\leq g_j\leq f_{i(j)}$, we have that $$\inf_{\vec h}\psi_2^{k,m}(\vec f,\vec g'',\vec h)=0.$$  It follows that $(\dagger)$ holds and the proof is concluded.

\end{proof}

By a \emph{condition} we mean a finite set of expressions of the form $\varphi(\vec x)<r$ where $\varphi(\vec x)$ is a quantifier-free formula and $r\in \r$.  If $A$ is a C$^*$-algebra and $\vec a$ is a tuple from $A$, we say that $\vec a$ satisfies the condition $p(\vec x)$ if $\varphi(\vec a)^A<r$ for all expressions $\varphi(\vec x)<r$ belonging to $p(\vec x)$.

The following fact can be proven using model-theoretic forcing; see, for example, \cite[Appendix A]{GS}.

\begin{fact}\label{forcing}
Suppose that $\mathcal{K}$ is a class of separable models of $T_{\conn}$ that is uniformly definable by a sequence of universal types as witnessed by the formulae $(\varphi_{m,n}(\vec x_m))$.  Further suppose that, for every $\epsilon>0$, every $m\in \n$ and every satisfiable condition $p(\vec x)$ there is a model $C(X)$ of $T_{\conn}$ and $\vec f\in C(X)$ that satisfies $p(\vec x)$ and for which $\inf_n \varphi_{m,n}(\vec f)^{C(X)}<\epsilon$.  Then there is a separable existentially closed model of $T_{\conn}$ that belongs to the class $\mathcal{K}$.
\end{fact}

\begin{cor}\label{pseudoec}
There is a separable existentially closed model of $T_{\conn}$ that is chainable, which is thus necessarily isomorphic to $C(\mathbb{P})$.
\end{cor}

\begin{proof}
Suppose that $p(\vec x)$ is a condition that is satisfied in $C(X)$ for some continuum $X$.  Using Fact \ref{kp}, we may embed $C(X)$ into $C(Y)$ with $C(Y)\equiv C(\mathbb{P})$, whence it follows that $p(\vec x)$ is satisfied in $C(Y)$ and hence in $C(\mathbb{P})$.  In particular, for any $m,k\geq 1$, we have $\vec f\in C(\mathbb{P})$ such that $\vec f$ satisfies $p(\vec x)$ and $\sigma_k(\vec f)<\frac{1}{m}$.  Thus, we can apply Fact \ref{forcing} to obtain a separable existentially closed model of $T_{\conn}$ that is chainable.
\end{proof}

\begin{nrmk}
It is known that $\mathbb{P}$ is generic in the descriptive set-theoretic sense, that is, in the space of subcontinua of $[0,1]^\n$, the set of those continua homeomorphic to $\mathbb{P}$ is a dense $G_\delta$ set.  One can view Corollary \ref{pseudoec} as the statement that the pseudoarc is also model-theoretically generic as it arises as the generic model constructed using model-theoretic forcing.
\end{nrmk}

\begin{question}
What other properties of (metrizable) continua are uniformly definable by a sequence of universal types?
\end{question}

By the same arguments as above, if (P) is any property of metrizable continua that is uniformly definable by a sequence of universal types, then we can find a separable model $C(X)$ of $T_{\conn}$ that is existentially closed, chainable, and has property (P); since such a $C(X)$ is  necessarily isomorphic to $C(\mathbb{P})$, we conclude that $\mathbb{P}$ has property (P).  This could be a potentially new way of establishing continuum-theoretic properties of the pseudoarc.

\section{Connection to Quantifier Elimination}

In \cite{EFKV}, the following questions are posed:

\begin{question}
Does there exist a model $C(X)$ of $T_{\conn}$ for which $C(X)$ admits quantifier elimination and $C(X)\neq\mathbb C$?  In particular, does $C(\mathbb{P})$ admit quantifier-elimination?
\end{question}

In this section we answer both of these questions in the negative.  We first recall that a theory $T^*$ is a \emph{model companion} of $T_{\conn}$ if $T^*$ is a model-complete theory such that every model of $T_{\conn}$ embeds into a model of $T^*$ and vice-versa; if $T_{\conn}$ has a model companion, then it is necessarily unique (up to logical equivalence).

\begin{lemma}\label{prop:ModelCompanion}
Suppose that $C(X)\models T_{\conn}$, $X$ is nondegenerate, and $C(X)$ is model-complete.  Then $\Th(C(X))$ is the model-companion of $T_{\conn}$.
\end{lemma}

\begin{proof}
It suffices to show that if $C(Y)\models T_{\conn}$, then $C(Y)$ embeds into a model of $\Th(C(X))$.  However, this follows immediately from Fact \ref{kp}.
\end{proof}

We recall that if $T^*$ is the model companion of $T_{\conn}$, then $T^*$ is said to be the \emph{model completion} of $T_{\conn}$ if $T^*$ has quantifier-elimination.  (This is not the official definition of model completion, but is convenient for our purposes here.)  

\begin{prop}\label{prop:ModelCompletion}
$T_{\conn}$ does not have a model completion.
\end{prop}

\begin{proof}
It is a standard fact of model theory that if $T_{\conn}$ has a model-completion, then $T_{\conn}$ has the amalgamation property.  Stated in terms of continua, this means that whenever $X$, $Y$, and $Z$ are continua and $f:X\to Z$ and $g:Y\to Z$ are continuous surjections, there is a continuum $W$ and continuous surjections $r:W\to X$ and $s:W\to Y$ such that $f\circ r=g\circ s$.  We give an example, due to Logan Hoehn, to show that the class of continua does not enjoy this co-amalgamation property.

Let $X = Y = [0, 1]$, and let $Z$ be the circle $\mathbb{S}^1$, which, for convenience, we view as the subset of $\mathbb{C}$ consisting of complex numbers $e^{i\theta}$.  Let $f : X \to Z$ be $f(x) = e^{2\pi i x}$ and $g : Y \to Z$ be $g(y) = e^{\pi i (2y + 1)}$.  Suppose that $W, r, s$ complete the amalgamation, in the above sense.  Let $A = r^{-1}([0, 1/2]) \cap s^{-1}([1/2, 1])$ and $B = r^{-1}([1/2, 1]) \cap s^{-1}([0, 1/2])$, both of which are closed in $W$.  It is easy to see that $A \cup B = W$.  Now suppose that $w \in A \cap B$.  Then $r(w) = s(w) = 1/2$, so $f(r(w)) = e^{\pi i} \neq e^{2\pi i} = g(s(w))$, contradicting the assumption that $f \circ r = g \circ s$.  Therefore $A \cap B = \emptyset$, and so $W$ is disconnected, yielding a contradiction.
\end{proof}

\begin{rmk}
Although the class of continua does not satisfy the amalgamation property in general, there are some continua that can always be amalgamated over.  For example, in \cite{Kras}, Krasinkiewicz shows that one can always amalgamate over an arc.  It is a standard fact of model theory (see \cite{ecfactor} for the proof in general or \cite{bankston2} for a proof in the case of continua) that one can always amalgamate over existentially closed structures.  In particular, as a consequence of Corollary \ref{pseudoec}, we see that one can always amalgamate over the pseudoarc.
\end{rmk}

\begin{cor}
If $X$ is a nondegenerate continuum, then $C(X)$ does not have quantifier-elimination.
\end{cor}

\begin{proof}
If $C(X)$ had quantifier-elimination, then by Lemma \ref{prop:ModelCompanion}, $\Th(C(X))$ is the model-companion, and hence model-completion, of $T_{\conn}$, contradicting Proposition \ref{prop:ModelCompletion}.
\end{proof}

\begin{rmk}
In \cite{EFKV}, the authors show that if $C(X)$ has quantifier-elimination, then either $X$ is connected or else $C(X)$ is elementarily equivalent to $\mathbb{C}$, $\mathbb{C}^2$, or $C(2^\n)$.  The authors also ask if the former case can occur; our observations show that it cannot.
\end{rmk}

It is natural to wonder which abelian C$^*$-algebras are model-complete.  We remark here that there is only one possible connected infinite-dimensional abelian C$^*$-algebra that could be model-complete:

\begin{prop}
The following are equivalent:
\begin{enumerate}
\item $C(\mathbb{P})$ is model-complete;
\item there is a nondegenerate continuum $X$ such that $C(X)$ is model-complete;
\item $T_{\conn}$ has a model-companion.
\end{enumerate}
If these conditions hold, then $\Th(C(\mathbb{P}))$ is the model companion of $T_{\conn}$ and $C(\mathbb{P})$ is the unique (up to elementary equivalence) connected abelian C$^*$-algebra that is model-complete.
\end{prop}

\begin{proof}
The implication (1)$\Rightarrow$(2) is obvious and the implication (2)$\Rightarrow$(3) is the content of Lemma \ref{prop:ModelCompanion}.  For the direction (3)$\Rightarrow$(1), observe that if $T^*$ is the model-companion of $T$, then, since the models of $T^*$ are precisely the existentially closed models of $T_{\conn}$, by Theorem \ref{pseudoec} we have $C(\mathbb{P})\models T^*$.  If $C(X)\models T^*$, then by Fact \ref{kp}, there is $C(X')\equiv C(X)$ such that $C(\mathbb{P})$ embeds into $C(X')$; since $C(X')\models T^*$ and $T^*$ is model-complete, we see that the embedding of $C(\mathbb{P})$ into $C(X')$ is elementary.  In particular, $C(\mathbb{P})\equiv C(X')\equiv C(X)$.  It follows that $T^*$ is complete, whence $T^*=\Th(C(\mathbb{P}))$ and $C(\mathbb{P})$ is model-complete.  We have already remarked, as a consequence of Fact \ref{kp}, that there is at most one connected abelian C$^*$-algebra that is $\forall\exists$-axiomatizable; since model-completeness implies $\forall\exists$-axiomatizability, the proof is complete.
\end{proof}

\appendix
\section*{Appendix: Bridging the gaps between terminologies}

Earlier work on the model theory of continua worked directly in the category of compacta by considering dual versions of concepts from ordinary model theory.  In particular, there is a notion of \emph{ultracoproduct} of a family of compact spaces.  The key fact is that if $(X_i\mid i\in I)$ is a family of compact spaces, $\u$ is an ultrafilter on $I$, and $\sum_\u X_i$ denotes the ultracoproduct of the $X_i$'s with respect to the ultrafilter $\u$, then we have a canonical isomorphism
$$C\left(\sum_\u X_i\right)\cong \prod_\u C(X_i),$$ where the right-hand side of the above display is the usual ultraproduct of C$^*$-algebras.

In order to explain the connection between Bankston's notion of co-existential closedness of compacta and the usual notion of existential closedness of abelian C$^*$-algebras, it will be useful to explain the ultraco\emph{power} construction.  Towards this end, recall that for any set $I$, $\beta I$ denotes the set of ultrafilters on $I$.  If $I'$ is another set and $f\colon I\to I'$ is any function, we obtain an induced map $\beta f\colon\beta I\to \beta I'$ by declaring, for $\u\in \beta I$ and $S\subseteq I'$, that $S\in \beta f(\u)$ if and only if $f^{-1}(S)\in \u$.

Now suppose that $X$ is a compact space, $I$ is an infinite set, and $\u\in \beta I$.  Let $p\colon X\times I\to X$ and $q\colon X\times I\to I$ denote the projections onto the first and second components, respectively.  The ultracopower $\sum_{\u}X$ is then defined to be $(\beta q)^{-1}(\{\u\})$.  The map $p_{X,\u}:=\beta p|\sum_{\u}X$ is a continuous surjection of $\sum_{\u}X$ onto $X$ that is dual to the usual diagonal embedding, that is, the induced map $C(X)\hookrightarrow C(\sum_{\u}X)\cong C(X)^\u$ is the usual diagonal embedding $\Delta_{C(X)}$ of $C(X)$ into $C(X)^\u$.

Inspired by the Keisler-Shelah theorem of model theory, Bankston defines two compacta to be \emph{co-elementarily equivalent} if they have homeomorphic ultracopowers.  It follows from the relationship between ultrapowers of C${}^*$-algebras and ultracopowers of spaces mentioned above that compacta $X$ and $X'$ are co-elementarily equivalent if and only if $C(X)$ and $C(X')$ are elementarily equivalent.  Before the advent of continuous model theory, the literature on the model theory of continua connects with classical model theory via lattices that are bases for the collection of closed subsets of continua.   In particular, if $L$ and $L'$ are lattices which are bases for the closed subsets of $X$ and $X'$ respectively and if $L$ and $L'$ are elementarily equivalent in the sense of ordinary first-order logic, then $X$ and $X'$ are co-elementarily equivalent (see \cite{bankston8}).  This observation allows for the translation of \cite[Proposition 3.1]{hart}, which is stated in the language of bases of closed sets, to the statement of Fact \ref{kp} above.

Bankston defines a continuous surjection $f\colon X\to Y$ between compacta to be \emph{co-existential} if there is an ultracopower $\sum_{\u}Y$ of $Y$ and a continuous surjection $g\colon \sum_{\u}Y\to X$ such that $f\circ g=p_{Y,\u}$.  

Suppose that $f\colon X\to Y$ is co-existential.  By considering the induced maps on the C$^*$-algebras, we get embeddings $f^\#\colon C(Y)\hookrightarrow C(X)$ and $g^\#\colon C(X)\hookrightarrow C(Y)^\u$ for which $(f\circ g)^\#=g^\#\circ f^\#=\Delta_{C(Y)}$.  Since $\Delta$ is elementary, it follows that $f^\#$ is an existential embedding.  

Conversely, suppose that $f\colon X\to Y$ is a continuous surjection for which $f^\#\colon C(Y)\hookrightarrow C(X)$ is existential.  Then for an ultrafilter $\u$ that makes $C(Y)^\u$ sufficiently saturated, we can find an embedding $C(X)\hookrightarrow C(Y)^\u$ that restricts to the diagonal embedding of $C(Y)$ into $C(Y)^\u$.  If one lets $g\colon \sum_{\u}Y\to X$ be the continuous surjection that induces the aforementioned embedding of $C(X)$ into $C(Y)^\u$, we have that $g$ witnesses that $f$ is co-existential.  We have thus proven:

\begin{lemma*}
A continuous surjection $f:X\to Y$ is co-existential if and only if the induced map $f^\#:C(Y)\hookrightarrow C(X)$ is existential.
\end{lemma*}  

Bankston defines a compactum (resp. continuum) $Y$ to be \emph{co-existentially closed} if every continuous surjection $X\to Y$ with $X$ a compactum (resp. continuum) is co-existential; equivalently, $C(X)$ is an existentially closed abelian C$^*$-algebra (resp. existentially closed model of $T_{\conn}$).

In \cite{EV}, the first and third author of this paper proved that $C(2^\n)$ has quantifier elimination.  We take the opportunity here to point out that one can deduce this fact from the literature.  In \cite{bankston1}, Bankston proves that the co-existentially closed compacta are precisely the totally disconnected spaces without isolated points.  It follows that the existentially closed abelian C$^*$-algebras are precisely the models of $\Th(C(2^\n))$, whence $\Th(C(2^\n))$ is the model-companion of the theory of abelian C$^*$-algebras.  Since the theory of abelian C$^*$-algebras has the amalgamation property (by the existence of fiber products of compacta), the model companion is a model completion, whence $\Th(C(2^\n))$ has quantifier-elimination.

\bibliographystyle{amsplain}
\bibliography{pseudoarccoec}
\end{document}